\documentclass[12pt]{article}

\usepackage{amsfonts}
\usepackage{amsmath}
\usepackage{amsthm}

\newtheorem{lem}{Lemma}
\newtheorem{thm}[lem]{Theorem}

\newtheorem{cor}[lem]{Corollary}

\theoremstyle{definition}
\newtheorem{defn}[lem]{Definition}
\newtheorem{rem}[lem]{Remark} 

\newcommand{\N}{\ensuremath {\mathbb{N}}}

\newcommand{\calU} {\ensuremath {\mathcal{U}}}
\newcommand{\calP} {\ensuremath {\mathcal{P}}}
\newcommand{\calS} {\ensuremath {\mathcal{S}}}

\title{Exponential triples}
\author{Alessandro Sisto\\
\small Mathematical Institute, 24-29 St Giles, Oxford OX1 3LB, United Kingdom\\
\small \texttt{sisto@maths.ox.ac.uk}
}

\date{}

\begin{document}
\maketitle

\begin{abstract}
Using ultrafilter techniques we show that in any partition of $\mathbb{N}$ into 2 cells there is one cell containing infinitely many exponential triples, i.e. triples of the kind $a,b,a^b$ (with $a,b>1$). Also, we will show that any multiplicative $IP^*$ set is an ``exponential $IP$ set'', the analogue of an $IP$ set with respect to exponentiation.
 
\end{abstract}

\section*{Introduction}

A well-known theorem by Hindman states that given any finite partition of $\N$, there exists an infinite sets $X$ and one cell of the partition containing the finite sums of $X$ (and also the finite products of some infinite set $Y$), see \cite{Hi}. Ultrafilters can be used to give a simpler proof than the original one, see \cite{Be}\footnote{This is available on Bergelson's webpage http://www.math.osu.edu/$\sim$vitaly/}, \cite{HS}.
\par
We will be interested in similar results involving exponentiation instead of addition and multiplication, and our methods of proof will involve ultrafilter arguments. The first main result of this paper is the following.

\begin{thm}
Consider a partition of the natural numbers $\N=A\cup B$. Either $A$ or $B$ contains infinitely many triples $a,b,a^b$, with $a,b>1$.
\end{thm}

Next, we will provide results (Theorems \ref{fegen1} and \ref{fegen2}) which allow to find larger structures than the triples as above inside multiplicative $IP^*$ sets (see Definition \ref{ip*}). A corollary of those theorems (Corollary \ref{fecor}) is given below.

\begin{defn}\label{fe:defn}
Consider an infinite set $X\subseteq \N$ and write $X=\{x_i\}_{i\in\N}$, with $x_j<x_{j+1}$ for each $j$. Define inductively
$$FE^I_{n+1}(X)= \{y^{x_{n+1}}|y\in FE^I_n(X)\}\cup FE^I_{n}(X) \cup\{x_{n+1}\},$$
$$FE^{II}_{n+1}(X)=\{(x_{n+1})^y|y\in FE^{II}_n(X)\}\cup FE^{II}_{n}(X) \cup\{x_{n+1}\},$$
with $FE^I_0(X)=FE^{II}_0(X)=\{x_0\}$. Set
$$FE^I(X)=\bigcup_{n\in\N} FE^I_n(X),$$
$$FE^{II}(X)=\bigcup_{n\in\N} FE^{II}_n(X).$$
We will say that $C\subseteq \N$ is an exponential $IP$ set of type I (resp. II) if it contains a set $FE^I(X)$ (resp. $FE^{II}(X)$) for some infinite $X$.
\end{defn}

As usual, $FS(C)$ and $FP(C)$ denote the set of finite sums and finite products of $C\subseteq \N$ (see definition \ref{fsfp}).

\begin{thm}
 Given any multiplicative $IP^*$ set $A$ there exist some infinite $X,Y\subseteq \N$ such that $FS(X), FE^I(X), FP(Y), FE^{II}(Y)\subseteq A$.
\end{thm}

\subsection*{Acknowledgement}
The author thanks Mauro Di Nasso for suggesting the problem and the referee for very helpful recommendations.

\section{Preliminaries}

This section contains all the results about ultrafilters we will need. The reader is referred to \cite{Be} and \cite{HS} for further details.

An ultrafilter $\calU$ on $\N$ is a collection of subsets of $\N$ such that
\begin{enumerate}
 \item $\N\in \calU$,
 \item $A,B\in\calU\Rightarrow A\cap B\in\calU$,
 \item $A\in\calU, B\supseteq A\Rightarrow B\in\calU$,
 \item $\forall A\subseteq \N$ either $A\in\calU$ or $A^c\in\calU$.
\end{enumerate}

The set of all ultrafilters on $\N$ is denoted by $\beta\N$. Notice that it contains a copy of $\N$: given any $n\in\N$ the collection of subsets of $\N$ containing $n$ is an ultrafilter. The sum and product on $\N$ can be extended\footnote{Indeed, there are two ways to do this, and the one used in \cite{Be} is not the same as the one used in \cite{HS}. This will not affect what follows.} to $\beta\N$ to operations that we will still denote by $+$ and $\cdot$ (they are not commutative). We have that $(\beta\N,+)$ and $(\beta\N,\cdot)$ are semigroups.
\par
Given a semigroup $(S,*)$, an idempotent in $(S,*)$ is $s\in S$ such that $s*s=s$. Idempotent ultrafilters are of interest to us because of the result stated below.

\begin{defn}
\label{fsfp}
Let $X$ be a subset of $\N$. Denote
$$FS(X)=\left\{\sum_{i=0}^n x_i|n\in\N, x_i\in X, x_0<\dots<x_n\right\}$$
and
$$FP(X)=\left\{\prod_{i=0}^n x_i|n\in\N, x_i\in X, x_0<\dots<x_n\right\}.$$
We will say that $C\subseteq \N$ is an additive (resp. multiplicative) $IP$ set if it contains a set $FS(X)$ (resp. $FP(X)$) for some infinite $X$.
\end{defn}

\begin{thm}[\cite{HS}, Theorem 5.8, Lemma 5.11]
\label{ultrweneed}
 $(\beta\N,+)$ and $(\beta\N,\cdot)$ contain idempotent ultrafilters. Also, if $\calU$ is idempotent in $(\beta\N,+)$ (resp. $(\beta\N,\cdot)$) then any $U\in\calU$ is an additive (resp. multiplicative) $IP$ set. What is more, given any sequence $\{x_n\}_{n\in\N}$ there exists an ultrafilter idempotent in $(\beta\N,+)$ (resp. $(\beta\N,\cdot)$) such that for each $m\in\N$, $FS(\{x_n\}_{n\geq m})\in\calU$ (resp. $FP(\{x_n\}_{n\geq m})\in\calU$).
\end{thm}

We will also need the following.

\begin{thm}\label{alapult}
 There exists $\calU\in\beta\N$ such that each $U\in\calU$ contains arbitrarily long (non-trivial) geometric progressions.
\end{thm}

\begin{proof}
 The theorem follows from \cite[Theorem 5.7]{HS} together with (a corollary of) van der Waerden's Theorem that given any finite partition of $\N$ there is one cell containing arbitrarily long geometric progressions. (The usual van der Waerden's Theorem gives a cell containing arbitrarily long \emph{arithmetic} progressions, but one can deduce the stated result considering the restriction of the partition to $\{2^n:n\in\N\}$.)
\end{proof}

(The theorem can also be proven considering minimal idempotent ultrafilters.)
\par
We will use the following notation.

\begin{defn}
 If $A\subseteq \N$ and $n\in\N$ denote
\begin{enumerate}
 \item $-n+A=\{m\in\N:m+n\in A\}$,
 \item if $n\geq 1$, $(n^{-1})A=\{m\in\N:m\cdot n\in A\}$,
 \item if $n\geq 2$, $\log_n[A]=\{m\in\N: n^m\in A\}$,
 \item if $n\geq 1$, $A^{1/n}=\{m\in\N: m^n\in A\}$.
\end{enumerate}

\end{defn}

\begin{defn}
 Fix an ultrafilter $\calU$ on $\N$ and let $A\subseteq \N$. Set
$$A^\star_+=\{x\in A:-x+A\in\calU\},$$
$$A^\star_\bullet=\{x\in A:(x^{-1})A\in\calU\}.$$
\end{defn}

\begin{lem}[\cite{HS}, Lemma 4.14]
 If $\calU+\calU=\calU$ (resp. $\calU\cdot\calU=\calU$) and $A\in\calU$, then $A^\star_+\in\calU$ (resp. $A^\star_\bullet\in\calU$).
\end{lem}

\section{Exponential triples and exponential IP sets}

\subsection{Exponential triples}

\begin{defn}
An exponential triple is an ordered triple of natural numbers $(a,b,c)$ such that $a^b=c$ and $a,b>1$, to avoid trivialities. We will say that $C\subseteq \N$ contains the exponential triple $(a,b,c)$ if $a,b,c\in C$.

\end{defn}

\begin{thm}
\label{triples}
Consider a partition of the natural numbers $\N=A\cup B$. Either $A$ or $B$ contains infinitely many exponential triples.
\end{thm}

\begin{proof}
Let $\calU$ be an ultrafilter as in Theorem \ref{alapult}. Up to exchanging $A$ and $B$, we can assume $A\in\calU$.
Set, for each $n\geq 2$, $A_n=\log_n[A]\cap A$ and $B_n=\log_n[B]\cap A$. If $A_n\in\calU$ for each $n>1$, then clearly $A$ contains infinitely many exponential triples (if $a\in A$ and $b\in A\cap \log_a[A]$ then $\{a,b,a^b\}\subseteq A$).
\par
If this is not the case, consider some $n>1$ such that $B_n\in\calU$. Consider a geometric progression $a,ah,\dots, ah^k$ contained in $B_n$ (with $a>0, h>1$). If $h^i\in B$ for some $i\in\{1,\dots,k\}$, we have that the exponential triple $(n^a, h^i,n^{ah^i})$ is contained in $B$. 
If there are infinitely many geometric progressions $a,ah,\dots,ah^k$ ($a>0, h>1$) contained in $B_n$ and such that $h^i\in B$ for some $i\in\{1,\dots,k\}$, by the argument above it is readily seen that $B$ contains infinitely many exponential triples.
\par
Finally, if this does not hold we have that $A$ contains arbitrarily long progressions of the kind $h,h^2,\dots,h^k$. It is clear in this case that $A$ contains infinitely many exponential triples.

\end{proof}

\subsection{Exponential IP sets}

\begin{defn}
 \label{ip*}
 An additive (resp. multiplicative) $IP^*$ set is a set whose complement is not an additive (resp. multiplicative) $IP$ set.
\end{defn}

\begin{lem}
\label{prelim}
 Let $A$ be a multiplicative $IP^*$ set and let $n\in\N$.
\begin{enumerate}
 \item If $n\geq 2$ then $\log_n[A]$ is an additive $IP^*$ set.
 \item If $n\geq 1$, $A^{1/n}$ is a multiplicative $IP^*$ set.
\end{enumerate}

\end{lem}

\begin{proof}
 $1)$ Consider $FS(X)$ for some infinite $X$. We have to show $\log_n[A]\cap FS(X)\neq\emptyset$. As $A$ is a multiplicative $IP^*$ set we have $FP(n^X)\cap A\neq \emptyset$, which clearly implies $\log_n[A]\cap FS(X)\neq\emptyset$.
\par
 $2)$ Consider $FP(X)$ for some infinite $X$. As $FP(X^n)\cap A\neq \emptyset$, we have $A^{1/n}\cap FP(X)\neq\emptyset$.

\end{proof}

The following theorem is inspired by \cite[Theorem 16.20]{HS}, and the proof closely follows the proof of that theorem. (A simpler proof of a simpler fact will be given in Remark \ref{simpler}.) We will denote the collection of finite subsets of $\N$ by $\calP_f(\N)$.

\begin{thm}
\label{fegen1}
 Let $\calS$ be the set of finite sequences in $\N$ (including the empty sequence)
and let $f:\calS \to \N$. Let $\{y_n\}_{n\in\N}\subseteq \N$ be a sequence and let $A$ be a multiplicative $IP^*$ set. Then there exists $\{x_n\}_{n\in\N}$ such that $FS(\{x_n\}_{n\in\N})\subseteq FS(\{y_n\}_{n\in\N})$ and whenever $F\in\calP_f(\N)$, $l=f((x_0,\dots,x_{\min F-1}))$ and $t\in\{2,\dots,l\}$ we have $t^{\sum_{j\in F} x_j}\in A$.
\end{thm}

\begin{proof}
In this proof we set $C^\star=C^\star_+$ for each $C\subseteq \N$. Let $\calU$ be an ultrafilter idempotent in $(\beta\N,+)$ such that $FS(\{y_n\}_{n\geq m})\in\calU$ for each $m\in\N$; see Theorem \ref{ultrweneed}. By the previous lemma, for each $t\in\N$ with $t\geq 2$ we have that $\log_t[A]$ is an additive $IP^*$ set, and therefore $\log_t[A]\in \calU$. In particular, we have $B_0\in \calU$, where $$B_0=FS(\{y_n\}_{n\in\N})\cap\bigcap_{t=2}^{f(\emptyset)} \log_t[A].$$
Pick any $x_0\in B_0^\star$ and $H_0\in\calP_f(\N)$ such that $x_0=\sum_{t\in H_0} y_t$.
\par
We will choose inductively $x_i,H_i$ and $B_i$ satisfying the following properties.
\begin{enumerate}
 \item $x_i=\sum_{t\in H_i} y_t$,
 \item if $i\geq 1$ then $\min H_i> \max H_{i-1}$,
 \item $B_i\in\calU$,
 \item for each $\emptyset\neq F \subseteq \{0,\dots,i\}$ and $m=\min F$ we have $\sum_{j\in F} x_j\in B_m^\star$,
 \item if $i\geq 1$, then $B_i\subseteq \bigcap_{t=2}^{f((x_0,\dots,x_{i-1}))} \log_t[A]$.
\end{enumerate}

Those properties are satisfied for $x_0, H_0, B_0$ chosen as above. Let us now perform the inductive step: suppose that we have $x_i,H_i$ and $B_i$ for $i\leq n$ satisfying the required properties.
\par
Set $k=\max H_n +1$. By our choice of $\calU$, we have $FS(\{y_t\}_{t\geq k})\in\calU$. Set, for $m\leq n$,
$$E_m=\left\{\sum_{j\in F} x_j: \emptyset\neq F\subseteq \{0,\dots,n\} \text{\ and\ } m=\min F\right\}.$$
By $(4)$, $E_m\subseteq B_m^\star$ for each $m\leq n$, so that for every $a\in E_m$ we have $-a+B_m^\star \in\calU$ by Lemma \ref{prelim}. We can then set
$$B_{n+1}=FS(\{y_t\}_{t\geq k})\cap\bigcap_{t=2}^{f((x_0,\dots,x_n))} \log_t[A]\cap\bigcap_{m\leq n}\bigcap_{a\in E_m}(-a+B_m^\star),$$
and we have $B_{n+1}\in\calU$. Pick any $x_{n+1}\in B_{n+1}^\star$ and choose $H_{n+1}\in\calP_f(\N)$ with $\min H_{n+1}\geq k$ and $x_{n+1}=\sum_{t\in H_{n+1}} y_t$.
\par
We only need to check $(4)$.
\par
Let $\emptyset\neq F \subseteq \{0,\dots,n+1\}$ and set $m=\min F$. We have to show that $\sum_{j\in F} x_j\in B_m^\star$. If $n+1\notin F$, the conclusion holds by the inductive hypothesis. Also, if $F=\{n+1\}$ then $m=n+1$ and $\sum_{j\in F} x_j=x_{n+1}\in B_{n+1}^\star=B^\star_m$. So, we can assume $n+1\in F$ and $G=F\backslash \{n+1\}\neq\emptyset$. Set $a=\sum_{j\in G} x_j$. As $G$ is non-empty, we have $a\in E_m$. Hence (as $B_{n+1}\subseteq -a+B_m^\star$) $x_{n+1}\in -a+B_m^\star$, that is to say $\sum_{j\in F} x_j=a+x_{n+1}\in B_m^\star$.
\par
We are now ready to complete the proof. Let $F,l,t$ be as in the statement. By $(4)$ and $(5)$ we get $\sum_{j\in F} x_j\in B_{\min F}\subseteq \log_t[A]$, which by definition means $t^{\sum_{j\in F} x_j}\in A$. Also, $(1)$ and $(2)$ guarantee that $FS(\{x_n\})\subseteq FS(\{y_n\})$.

\end{proof}

The following theorem can be proven in the same way, using a suitable ultrafilter idempotent in $(\N,\cdot)$ and $C^\star_\bullet$ instead of $C^\star_+$.

\begin{thm}
\label{fegen2}
 Let $\calS$ be the set of finite sequences in $\N$ and let $f:\calS \to \N$. Let $\{y_n\}_{n\in\N}\subseteq \N$ be a sequence and let $A$ be a multiplicative $IP^*$ set. Then there exists $\{x_n\}_{n\in\N}$ such that $FP(\{x_n\}_{n\in\N})\subseteq FP(\{y_n\}_{n\in\N})$ and whenever $F\in\calP_f(\N)$, $l=f((x_0,\dots,x_{\min F-1}))$ and $t\in\{1,\dots,l\}$ we have $({\prod_{j\in F} x_j})^t\in A$.
\end{thm}

We now give an application of the theorems above. One can define similar notions of exponential $IP$ set and obtain the corollary below using the same argument. Recall that we defined exponential $IP$ sets in the Introduction.

\begin{cor}
\label{fecor}
For any multiplicative $IP^*$ set $A$ there exists some infinite $X, Y\subseteq \N$ such that $FS(X),FE^I(X), FP(Y), FE^{II}(Y)\subseteq A$.
\end{cor}

\begin{proof}
Let $B=A^c$. The main result of \cite{Hi} (\cite[Corollary 5.22]{HS}) gives that one between $A$ and $B$ is both an additive and a multiplicative $IP$ set. But $B$ is not a multiplicative $IP$ set, hence $A$ is an additive $IP$ set.
 Let $\{y_i\}$ be such that $FS(\{y_i\}_{i\in\N})\subseteq A$, and define $f:\calS\to\N$ as $f((x_0,\dots,x_n))=\max FE^I_n(\{x_0,\dots,x_n\})$. Let $X=\{x_n\}_{n\in\N}$ be as in Theorem \ref{fegen1} (we can assume that each $x_i$ is greater than 1). Clearly, $FS(X)\subseteq A$.  We will show inductively $FE^I_i(X)\subseteq A$. Notice that $X\subseteq A$ (in particular $FE^I_0(X)\subseteq A$). Suppose $FE^I_{n}(X)\subseteq A$ and consider $y^{x_{n+1}}\in FE^I_{n+1}(X)$. As $2\leq y\leq f((x_0,\dots,x_n))$, Theorem \ref{fegen1} gives $y^{x_{n+1}}\in A$.
\par
 The set $Y$ can be found applying Theorem \ref{fegen2} in a similar way.
\end{proof}

\begin{rem}
\label{simpler}
 We now give a simpler proof that any multiplicative $IP^*$ set is an exponential $IP$ set of type I (a similar proof can be given for type II).
\par
 Let $A$ be a multiplicative $IP^*$ set and let $B=A^c$. As shown in the proof of the corollary, $A$ is an additive $IP$ set. Let $\calU$ be an idempotent ultrafilter in $(\beta\N,+)$ such that $A\in\calU$ (see Theorem \ref{ultrweneed}).
\par
For each $n\in\N$, $n>1$, we have that $B_n=\log_n[B]\cap A\notin\calU$, for otherwise $B_n$ would be an additive $IP$ set and $n^{B_n}\subseteq B$ would be a multiplicative $IP$ set. Therefore $A_n=\log_n[A]\cap A\in\calU$, for each $n>1$. We are ready to construct a set $X$ such that $FE^I(X)\subseteq A$. Just set $X=\{x_i\}_{i\in\N}$, for any sequence $\{x_i\}$ which satisfies:
$$\left\{
\begin{array}{l}
x_0\in A, x_0>1,\\
N_i=\max FE^I_i(\{x_j\}_{j\leq i}),\\
x_{i+1}\in \bigcap_{j=2}^{N_i} A_j\cap A.\\
\end{array}
\right. $$
\end{rem}

\subsection{Open questions}
There are several natural questions which arise at this point. For example, is there an elementary proof of Theorem \ref{triples}? Does it hold for partitions of $\N$ into finitely many cells? How about just 3 cells? Is it true that given any partition of $\N$ into 2 cells, one of the cells is an exponential $IP$ set (of type I and/or II)? How about finitely many cells?

\end{document}